\documentclass[12pt,draftcls,onecolumn]{IEEEtran}

\usepackage{amssymb}
\usepackage{amsmath}
\usepackage{graphicx}
\usepackage{fancyhdr}

\newtheorem{theorem}{Theorem}
\newtheorem{lemma}{Lemma}
\newtheorem{remark}{Remark}

\begin{document}

\title{An Asymptotically Stable Continuous Robust Controller for a Class of Uncertain MIMO Nonlinear Systems}

\author{Baris Bidikli, Enver Tatlicioglu, Erkan Zergeroglu, Alper Bayrak \thanks{B. Bidikli, E. Tatlicioglu and A. Bayrak are with the Department of Electrical \& Electronics Engineering, Izmir Institute of Technology, Izmir, 35430 Turkey (Phone: +90 (232) 7506536; Fax: +90 (232) 7506599; E-mail: [barisbidikli,envertatlicioglu,alperbayrak]@iyte.edu.tr). E. Zergeroglu is with the Department of Computer Engineering, Gebze Institute of Technology, 41400, Gebze, Kocaeli, Turkey (Email: ezerger@bilmuh.gyte.edu.tr).}}
\maketitle

\begin{abstract}
In this work, we propose the design and analysis of a novel continuous robust controller for a class of multi--input multi--output (MIMO) nonlinear uncertain systems. The systems under consideration contains unstructured uncertainties in their drift and input matrices. The proposed controller compensates the overall system uncertainties and achieves asymptotic tracking where only the sign of the leading minors of the input gain matrix is assumed to be known. A Lyapunov based argument backed up with an integral inequality is applied to prove stability of the proposed controller and asymptotic convergence of the error signals. Simulation results are presented in order to illustrate the viability of the proposed method.
\end{abstract}

\section{Introduction}
The problem of controller design for multi--input multi--output (MIMO) systems with uncertainties have been the research interest of many researchers. When the system under consideration is linear, solutions as early as \cite{Sastry-Boston} are available. To name a few, in \cite{Sastry-Boston}, assuming that the exact knowledge of high frequency input gain matrix is available, an adaptive controller was proposed. Ioannou and Sun \cite{Ioannou-Sun} used a less restrictive assumption on the gain matrix, the controller proposed required the existence of an auxiliary matrix which when pre--multiplied by the high frequency gain matrix the resultant matrix would be positive definite and symmetric. In \cite{Costa-Aut}, an adaptive controller for minimum phase systems with degree one has been proposed assuming that only the signs of the leading minors of the high frequency input gain matrix is known. For nonlinear MIMO uncertain systems the problem becomes more complicated. Therefore the solutions for only some special cases are considered in the literature. An adaptive backstepping method for strict feedback systems was utilized in \cite{KrsticBook} with the assumption that the input gain matrix premultiplying the control input is known. In \cite{Kosmatopoulos}, a general procedure for the design of switching adaptive controllers including feedback linearizable, and parametric--pure feedback systems has been proposed. An adaptive neural controller for MIMO systems with block--triangular form was proposed in \cite{GeWang-TNN}. Some more applications of controller design for MIMO uncertain systems like visual servoing, thermal management, aeroelasticity vibration suppression and ship control were presented in \cite{Zerger-TMech}, \cite{Setlur03}, \cite{ReddyChenBehalMarzocca} and \cite{ship-conf}, respectively. 

Recently, \cite{WangChenBehal}, \cite{WangBehalXianChen}, \cite{WangBehal} have proposed robust and adaptive type controllers for the MIMO nonlinear systems of the form 
\begin{equation}
x^{\left( n\right) }=H\left( x,\dot{x},\cdots ,x^{(n-1)}\right) +G\left( x,\dot{x},\cdots ,x^{(n-2)}\right) \tau   \label{model0}
\end{equation}
where $(\cdot)^{i} $ denotes the $i^{th} $ time derivative, $H\left(\cdot \right) \in \mathbb{R}^{m\times 1} $  and  $G\left(\cdot \right) \in \mathbb{R}^{m\times m}$ are first order differentiable uncertain functions with $G\left(\cdot \right) $ being a real matrix with non--zero leading principal minors. Specifically in  \cite{WangChenBehal}, authors have extended the work of \cite{ZhangAcc} by redesigning the controller of \cite{ChenBehalDawson} removing an algebraic loop and potential singularity in their previous design and obtained a global Uniformly Ultimately Bounded (UUB) tracking error performance. In \cite{WangBehalXianChen}, an adaptive controller that ensures asymptotic error tracking has been proposed. Recently, a continuous robust controller achieving semi--globally asymptotic tracking performance for uncertain MIMO systems of the form \eqref{model0} with two degrees of freedom was proposed in \cite{WangBehal}.

In this work, we consider a slightly broader class of uncertain MIMO nonlinear system than that of \eqref{model0} which is of the form   
\begin{equation}
x^{\left( n\right) }=h\left( X\right) +g\left( X\right) \tau  \label{model1}
\end{equation}
where $X\left( t\right) =\left[ \begin{array}{cccc} x^T & \dot{x}^T & \cdots & \left( x^{\left( n-1\right) }\right) ^T \end{array} \right] ^T \in \mathbb{R}^{\left( mn \right) \times 1}$ is the combined state vector with $x^{\left(i\right) }\left( t\right) \in \mathbb{R}^{m\times 1}$ $i=0,...,n$, being states, $h\left( \cdot \right) \in \mathbb{R}^{m\times 1}$ is an uncertain signal, $g\left(\cdot \right) \in \mathbb{R}^{m\times m}$ is a real matrix with non--zero leading principal minors, and $\tau \left( t\right) \in \mathbb{R}^{m\times 1}$ is the control input.  To our best knowledge there are only few works on this model in the literature. Namely in \cite{Ioannou-TAC}, Xu and Ioannou considered the case where $g\left(X\right)$ is either positive or negative definite, and designed a neural networks based adaptive controller that ensured local convergence of the tracking error to a residual set. While in \cite{xian-journal}, Xian \textit{et al.} considered the case where $g\left(X\right)$ is positive definite and a robust controller containing the integral of the signum of the error term was designed to obtain semi--global asymptotic tracking. More recently in \cite{ChenOFB}, Chen \textit{et al.} proposed a robust controller fused with a feedforward compensation term that ensured ultimate boundedness of the tracking error. Output feedback version of \cite{ChenOFB} were then proposed in \cite{Chen06a}, and \cite{Chen06}. 

In this paper, we have extended the full state feedback version of \cite{ChenOFB} to obtain asymptotic tracking as opposed to UUB. Apart from this, the proposed control design does not contain, though it is quite possible to design a similar one, an extra feedforward compensation term. An output feedback version can also be obtained similar to that of \cite{Chen06a}, or \cite{Chen06}, however this result would possible lead an  UUB error tracking performance. The main novelty of the method is the use of integral inequalities in conjunction with an initial Lyapunov based analysis to prove the boundedness of the error signals and then use the boundedness of the error terms in the final analysis to obtain the asymptotic tracking. The rest of the paper is organized as follows; Section II introduces the error system development while the full state controller development is introduced in Section III. Stability of the closed system under the proposed method and the numerical simulations are given in Sections IV and V, respectively. Finally, concluding remarks are presented in Section VI.

\section{Error System Development}

Based on the assumption that $g\left( \cdot \right) $ being a real matrix with non--zero leading principal minors, the following matrix decomposition is utilized \cite{Costa-Aut}, \cite{Morse-SCL}
\begin{equation}
g =S(X) D U(X) \label{decomposition}
\end{equation}
where $S\left( X\right) \in \mathbb{R}^{m\times m}$ is a symmetric positive definite matrix, $D \in \mathbb{R}^{m\times m}$ is a diagonal matrix with entries being $\pm 1$, and  $U\left( X\right)\in \mathbb{R}^{m\times m}$ is a unity upper triangular matrix. Similar to \cite{ChenOFB}, we assume that $D$ is available for control design. As we assumed that the leading principal minors of $g\left( X \right)$ are non--zero, from \eqref{model1} it is straightforward to obtain 
\begin{equation}
\tau =g^{-1}\left( x^{\left( n\right) }-h\right).  \label{tau1}
\end{equation}
Taking the time derivative of the system model in \eqref{model1} and then substituting \eqref{tau1} results 
\begin{eqnarray}
x^{\left( n+1\right) } &=&\dot{h}+\dot{g}\tau +g\dot{\tau}  \nonumber \\
&=&\dot{h}+\dot{g}g^{-1}\left( x^{\left( n\right) }-h\right) +g\dot{\tau} \nonumber \\
&=&\varphi +SDU\dot{\tau} \label{model2}
\end{eqnarray}
where \eqref{decomposition} was utilized and $\varphi \left( X,x^{\left(n\right) }\right) \in \mathbb{R}^{m\times 1}$ is an auxiliary signal defined to have the following form  
\begin{equation}
\varphi =\dot{h}+\dot{g}g^{-1}\left( x^{\left( n\right)} -h\right) . \label{phi}
\end{equation}
Multiplying both sides of \eqref{model2} with $S^{-1}\left( X\right)$ results in 
\begin{equation}
S^{-1}x^{\left( n+1\right) }=S^{-1}\varphi +DU\dot{\tau} \label{model3}
\end{equation}
and after defining $M\left( X\right) \triangleq S^{-1}\in \mathbb{R}^{m\times m}$ and $f\left( X,x^{\left( n\right) }\right) \triangleq S^{-1}\varphi \in \mathbb{R}^{m\times 1}$, we obtain 
\begin{equation}
Mx^{\left( n+1\right) }=f+DU\dot{\tau}  \label{model4}
\end{equation}
where  $M\left( X\right)$  is assumed to satisfy the following inequality
\begin{equation}
\underline{m}\left\Vert \chi \right\Vert ^{2}\leq \chi ^{T}M\left( X\right)
\chi \leq \bar{m}\left( X\right) \left\Vert \chi \right\Vert ^{2}~~~\forall
\chi \in \mathbb{R}^{m\times 1}  \label{M}
\end{equation}
with  $\underline{m}\in \mathbb{R}$ is a positive bounding constant, and $\bar{m}\left( X\right) \in \mathbb{R}$ is a
positive, non--decreasing function.

Our main control objective is to ensure that the system states would track a given smooth desired trajectory as closely as possible. In order to quantify the control objective, the output tracking error, $e_1\left( t\right) \in \mathbb{R}^{m\times 1}$, is defined as the difference between the reference and the actual system states as 
\begin{equation}
e_1\triangleq x_r-x \label{e1}
\end{equation}
where $x_r\left( t\right) \in \mathbb{R}^{m\times 1}$ is the reference trajectory satisfying following properties 
\begin{equation}
x_r\left( t\right) \in \mathcal{C}^n \text{ , } x_r^{\left( i\right) }\left( t\right) \in \mathcal{L}_\infty \text{ , }i=0,1,...,\left( n+1\right) .
\label{xds2}
\end{equation}
To ease the subsequent presentation, a combination of the reference trajectory and its time derivatives is also defined as $X_r\left( t\right) \triangleq \left[ \begin{array}{cccc} x_r^T & \dot{x}_r^T & \cdots & \left( x_r^{\left( n-1\right) }\right) ^T \end{array} \right] ^T\in \mathbb{R}^{\left(mn\right) \times 1}$. The control design objective is to develop a robust control law that ensures $\left\| e_1^{\left( i\right) }\left(t\right) \right\| \to 0$ as $t\to +\infty $, $i=0,...,n$, while ensuring that  all signals within the closed--loop system remain bounded. In our controller development, we will also assume that the combined state vector $X\left( t\right) $ is measurable.

To facilitate the control design, auxiliary error signals, denoted by $e_i\left( t\right) \in \mathbb{R}^{m\times 1}$, $i=2,3,...,n$, are defined as follows 
\begin{eqnarray}
e_2 & \triangleq & \dot{e}_1+e_1 \label{e2} \\
e_3 & \triangleq & \dot{e}_2+e_2+e_1 \label{e3} \\
& \vdots & \nonumber \\
e_n & \triangleq & \dot{e}_{n-1}+e_{n-1}+e_{n-2} . \label{en}
\end{eqnarray}
A general expression for $e_i\left( t\right) $, $i=2,3,...,n$ in terms of $e_1\left( t\right) $ and its time derivatives can be obtained as 
\begin{equation}
e_i=\sum\limits_{j=0}^{i-1}a_{i,j}e_1^{\left( j\right) }  \label{ei}
\end{equation}
where $a_{i,j}\in \mathbb{R}$ are known positive constants, generated via a Fibonacci number series. Our controller development also requires the definition of a filtered tracking error term, $r\left( t\right) \in \mathbb{R}^{m\times 1}$, defined to have the following form 
\begin{equation}
r\triangleq \dot{e}_n+\alpha e_n  \label{r}
\end{equation}
where $\alpha \in \mathbb{R}^{m\times m}$ is a constant positive definite, diagonal, gain matrix. After differentiating \eqref{r} and premultiplying the resulting equation with $M\left( \cdot \right) $, the following expression can be derived 
\begin{equation}
M\dot{r}=M\left( x_r^{\left( n+1\right)}+\sum\limits_{j=0}^{n-2}a_{n,j}e_1^{\left( j+2\right) }+\alpha \dot{e}_n\right) -f-DU\dot{\tau} \label{Mrdot}
\end{equation}
where \eqref{model4}, \eqref{e1}, \eqref{ei}, and the fact that $a_{n,(n-1)}=1$ were utilized. Defining the auxiliary function, $N\left( X,x^{\left( n\right) },t\right) \in \mathbb{R}^{m\times 1}$, as follows
\begin{equation}
N\triangleq M\left( x_r^{\left( n+1\right)}+\sum\limits_{j=0}^{n-2}a_{n,j}e_1^{\left( j+2\right) }+\alpha \dot{e}_n\right) -f+e_n+\frac 12\dot{M}r  \label{N}
\end{equation}
the expression in \eqref{Mrdot} can be reformulated to  have the following form 
\begin{equation}
M\dot{r}=-\frac 12\dot{M}r-e_n-DU\dot{\tau}+N . \label{Mrdot1}
\end{equation}
Furthermore, the filtered tracking error dynamics in \eqref{Mrdot1} can be rearranged as
\begin{equation}
M\dot{r}=-\frac 12\dot{M}r-e_n-D\left( U-I_m\right) \dot{\tau}-D\dot{\tau}+\widetilde{N}+\bar{N} \label{OpenLoop}
\end{equation}
where we added and subtracted $D\dot{\tau}\left( t\right) $ to the right--hand side, $I_m\in \mathbb{R}^{m\times m}$ is the standard identity matrix, and $\bar{N}\left( t\right) $, $\widetilde{N}\left( t\right) \in \mathbb{R}^{m\times 1}$ are auxiliary signals defined as follows
\begin{eqnarray}
\bar{N} &\triangleq & \left. N\right|_{X=X_r,x^{\left( n\right)}=x_r^{\left( n\right) }} \label{Nr} \\
\widetilde{N} &\triangleq & N-\bar{N} . \label{Ntilda}
\end{eqnarray}
The main idea behind adding and subtracting $D\dot{\tau}\left( t\right) $ term to the right--hand side of \eqref{OpenLoop}, is to make use of the fact that $U\left( \cdot \right) $ is unity upper triangular, and thus $\left(U-I_m\right) $ is strictly upper triangular. This property will later be utilized in the stability analysis.

\section{Controller Formulation}
Based on the open--loop error system in \eqref{OpenLoop} and the subsequent stability analysis, the control input, $\tau\left(t\right)$, is designed to have the following form  
\begin{equation}
\tau =DK\left[ e_n\left( t\right) -e_n\left( t_0\right) +\alpha \int_{t_0}^te_n\left( \sigma \right) d\sigma \right] +D\Pi  \label{u}
\end{equation}
where the auxiliary signal $\Pi \left( t\right) \in \mathbb{R}^{m\times 1}$ is generated according to 
\begin{equation}
\dot{\Pi}=C\text{Sgn}\left( e_n\right) ,\Pi \left( t_0\right) =0_{m\times 1}. \label{ub}
\end{equation}
In \eqref{u} and \eqref{ub}, $K$, $C\in \mathbb{R}^{m\times m}$ are constant, diagonal, positive definite, gain matrices, $0_{m\times 1}\in \mathbb{R}^{m\times 1}$ is a vector of zeros and Sgn$\left( \cdot \right) \in \mathbb{R}^{m\times 1}$ is the vector signum function. Based on the structures of \eqref{u} and \eqref{ub}, the following expression is obtained for the time derivative of the control input 
\begin{equation}
\dot{\tau}=DKr+DC\text{Sgn}\left( e_n\right) \label{udot}
\end{equation}
where \eqref{r} was utilized. The control gain is chosen as $K=I_m+k_pI_m+ \text{diag}\left\{ k_{d,1},...,k_{d,\left( m-1\right) },0\right\} $ where $k_p$, $k_{d,i} \in \mathbb{R}$ are constant, positive, control gains. Finally, after substituting \eqref{udot} into \eqref{OpenLoop}, the following closed--loop error system for $r\left( t\right) $ is obtained 
\begin{equation}
M\dot{r}=-\frac 12\dot{M}r-e_n-Kr+\widetilde{N}+\bar{N}-D\left( U-I_m\right)DKr-DUDC\text{Sgn}\left( e_n\right) \label{close}
\end{equation}
where the fact that $DD=I_m$ was utilized.

Before proceeding with the stability analysis, we would like to draw attention to the last two terms of \eqref{close} which we will investigate separately in the next two subsections: 
\subsection{The ``$D\left( U-I_m\right) DKr$'' Term}  
Note that, after utilizing the fact that $\left( U-I_m\right) $ being strictly upper triangular, we can rewrite the term $D\left( U-I_m\right) DKr$ as follows 
\begin{equation}
D\left( U-I_m\right) DKr=\left[ \begin{array}{c} \Lambda +\Phi \\ 0 \end{array} \right] \label{u1}
\end{equation}
where $\Lambda \left( t\right) $, $\Phi \left( t\right) \in \mathbb{R}^{\left( m-1\right) \times 1}$ are auxiliary signals with their entries $\Lambda_i\left( t\right) $, $\Phi _i\left( t\right) \in \mathbb{R}$, $i=1,...,\left( m-1\right) $, being defined as 
\begin{eqnarray}
\Lambda _i &=& d_i \sum\limits_{j=i+1}^m d_j k_j \widetilde{U}_{i,j} r_j  \label{u1a} \\
\Phi _i &=& d_i \sum\limits_{j=i+1}^m d_j k_j \bar{U}_{i,j} r_j  \label{u1b}
\end{eqnarray}
with $\bar{U}_{i,j}\left( X_r\right) $, $\widetilde{U}_{i,j}\left( t\right) \in \mathbb{R}$ are defined as
\begin{eqnarray}
\bar{U}_{i,j} & \triangleq &U_{i,j}|_{X=X_r} \label{Ubound1} \\
\widetilde{U}_{i,j} & \triangleq &U_{i,j}-\bar{U}_{i,j} \label{Ubound2}
\end{eqnarray}
where $U_{i,j}\left( X\right) \in \mathbb{R}$ are the entries of $U\left( X\right) $. Notice from \eqref{u1} that the last entry of the term $D\left(U-I_m\right) DKr$ is equal to $0$, and its $i$--th entry depends on the $\left( i+1\right) $--th to $m$--th entries of the control gain matrix $K$.

\subsection{The ``$DUDC\text{Sgn}\left( e_n\right) $'' Term}
We can rewrite the $DUDC\text{Sgn}\left( e_n\right) $ term as 
\begin{equation}
DUDC\text{Sgn}\left( e_n\right) =\left[ \begin{array}{c} \Psi \\ 0 \end{array} \right] + \Theta \label{u2}
\end{equation}
where $\Psi \left( t\right) \in \mathbb{R}^{\left( m-1\right) \times 1}$ and $\Theta \left( t\right) \in \mathbb{R}^{m\times 1}$ are auxiliary signals defined as 
\begin{eqnarray}
\left[ \begin{array}{c} \Psi \\ 0 \end{array} \right] &=& D\left( U-\bar{U}\right) DC\text{Sgn}\left( e_n\right) \label{u2a} \\
\Theta  &=&D\bar{U}DC\text{Sgn}\left( e_n\right) \label{u2b}
\end{eqnarray}
where $\bar{U}\left( X_r\right) =U|_{X=X_r}\in \mathbb{R}^{m\times m}$ is a function of reference trajectory and its time derivatives, and $\Psi_i\left( t\right) \in \mathbb{R}$, $i=1,...,\left(m-1\right)$ and $\Theta _i\left( t\right) \in \mathbb{R}$, $i=1,...,m$, are defined as 
\begin{eqnarray}
\Psi _i &=& d_i \sum\limits_{j=i+1}^m d_j C_j \widetilde{U}_{i,j} \text{sgn}\left( e_{n,j}\right) \label{Ubound3} \\
\Theta _i &=& d_i \sum\limits_{j=i}^m d_j C_j \bar{U}_{i,j} \text{sgn}\left( e_{n,j}\right) . \label{Ubound4}
\end{eqnarray}

\begin{remark}
The Mean Value Theorem \cite{khalil} can be utilized to develop the following upper bounds 
\begin{eqnarray}
\left\| \widetilde{N}\left( \cdot \right) \right\|  &\leq &\rho _{\widetilde{N}}\left( \left\| z\right\| \right) \left\| z\right\| \label{rho} \\
\left\| \widetilde{U}_{i,j}\left( \cdot \right) \right\|  &\leq &\rho_{i,j}\left( \left\| z\right\| \right) \left\| z\right\| \label{rho2}
\end{eqnarray}
where $\rho _{\widetilde{N}}\left( \cdot \right) $, $\rho _{i,j}\left( \cdot \right) \in \mathbb{R}$ are non--negative, globally invertible, non--decreasing functions of their arguments, and $z\left( t\right) \in \mathbb{R}^{\left[ \left( n+1\right) m\right] \times 1}$ is defined by 
\begin{equation}
z\triangleq \left[ \begin{array}{lllll} e_1^T & e_2^T & ... & e_n^T & r^T \end{array} \right] ^T.  \label{z}
\end{equation}
It can be seen from \eqref{xds2}, \eqref{N}, \eqref{Nr} that $\bar{N}\left( t\right) $ and $\bar{U}_{i,j}\left( t\right) $ are bounded in the sense that \cite{xian-journal} 
\begin{eqnarray}
\left| \bar{N}_i\left( t\right) \right|  &\leq &\zeta _{\bar{N}_i} \label{bound1} \\
\left| \bar{U}_{i,j}\left( t\right) \right|  &\leq &\zeta _{\bar{U}_{i,j}} \label{bound2}
\end{eqnarray}
where $\zeta_{\bar{N}_i}$, $\zeta_{\bar{U}_{i,j}}\in \mathbb{R}$ are positive bounding constants. Based on \eqref{u1a}, \eqref{u1b}, \eqref{Ubound3}, 
\eqref{Ubound4}, following upper bounds can be obtained 
\begin{eqnarray}
\left| \Lambda_i\right| &\leq &\sum\limits_{j=i+1}^mk_j\rho_{i,j}\left( \left\| z\right\| \right) \left\| z\right\| \left| r_j\right| \leq 
\rho_{\Lambda_i}\left( \left\| z\right\| \right) \left\| z\right\| \label{bound4} \\
\left| \Phi_i\right| &\leq &\sum\limits_{j=i+1}^m k_j \zeta _{\bar{U}_{i,j}} \left| r_j\right| \leq \zeta_{\Phi_i}\left\| z\right\| \label{bound3} \\
\left| \Psi_i\right| &\leq &\sum\limits_{j=i+1}^m C_j \rho_{i,j}\left( \left\| z\right\| \right) \left\| z\right\| \leq 
\rho_{\Psi_i}\left( \left\| z\right\| \right) \left\| z\right\| \label{bound5} \\
\left| \Theta_i\right| &\leq &\sum\limits_{j=i}^m C_j \zeta_{\bar{U}_{i,j}} \leq \zeta_{\Theta_i} \label{bound6}
\end{eqnarray}
where \eqref{rho}--\eqref{bound2} were utilized. From \eqref{bound6}, it is easy to see that $\left\| \Theta \right\| \leq \zeta_\Theta $ is satisfied for some positive bounding constant $\zeta_\Theta \in \mathbb{R}$, and from \eqref{bound4}--\eqref{bound5}, we have
\begin{equation}
\left| \Lambda_i\right| +\left| \Phi_i\right| +\left| \Psi_i\right| \leq \rho_i\left( \left\| z\right\| \right) \left\| z\right\| \label{bound7}
\end{equation}
where $\rho_i \left( \left\| z\right\| \right) \in \mathbb{R}$ $i=0,1,...,\left( m-1\right)$, are non--negative, globally invertible, non--decreasing functions satisfying 
\begin{equation}
\rho_{\Lambda_i}+\rho_{\Psi_i}+\zeta_{\Phi_{i}}\leq \rho_i. \label{bound8}
\end{equation}
\end{remark}

\begin{remark} \label{NrBound}
Notice that, as a result of the fact that $\bar{U}\left(t\right) $ being unity upper triangular, $\Theta \left( t\right) $ in \eqref{u2b} can be rewritten as 
\begin{equation}
\Theta =\left( I_m+\Omega \right) C\text{Sgn}\left( e_n\right) \label{bound}
\end{equation}
where $\Omega \left( t\right) \triangleq D\left( \bar{U}-I_m\right) D\in \mathbb{R}^{m\times m}$ is a strictly upper triangular matrix. Since it is a function of the reference trajectory and its time derivatives, its entries, denoted by $\Omega_{i,j}\left( t\right) \in \mathbb{R}$, are bounded in the sense that 
\begin{equation}
\left| \Omega_{i,j}\right| \leq \zeta_{\Omega_{i,j}} \label{bound9}
\end{equation}
where $\zeta_{\Omega_{i,j}} \in \mathbb{R}$ are positive bounding constants.
\end{remark}
At this point, we are now ready to continue with the stability analysis of the proposed robust controller. 

\section{Stability Analysis}
This section, via an initial Lyapunov based analysis  we will first prove the boundedness of the error signals under the closed--loop operation. Using this result  we will then present a lemma and obtain an upper bound for the integral of the absolute values of the entries of $e_n$. This upper bound will later be utilized in another lemma to prove the non--negativity of a Lyapunov--like function that will be used in our final analysis which proves asymptotic stability of the overall closed--loop system. 

\begin{theorem} \label{Theorem1}
\textbf{(Boundedness proof)} For the uncertain MIMO system of \eqref{model1}, the controller in \eqref{u} and \eqref{ub} guarantee the boundedness of the error signals  \eqref{e1}, \eqref{e2}--\eqref{en} and \eqref{r} provided that the control gains $k_{d,i}$ and $k_p$ are chosen large enough compared to the initial conditions of the system and the following condition is satisfied 
\begin{equation}
\lambda_{\min }\left( \alpha \right) \geq \frac 12 \label{GainBound1}
\end{equation}
where the notation $\lambda_{\min }\left( \alpha \right)$ denotes the minimum eigenvalue of the gain matrix $\alpha$, previously defined in \eqref{r}. 
\end{theorem}
\begin{proof}
See Appendix \ref{App1}.
\end{proof}

\begin{lemma} \label{Lemma1}
Provided that $e_n\left( t\right)$ and $\dot{e}_n\left( t\right)$ are bounded, the following expression for the upper bound of the integral of the absolute value of the $i$--th entry of $\dot{e}_n\left( t\right)$ $i=1,\cdots,m$ can be obtained
\begin{equation}
\int\limits_{t_0}^t\left| \dot{e}_{n,i}\left( \sigma \right) \right| d\sigma 
\leq \gamma _1+\gamma _2\int\limits_{t_0}^t\left| e_{n,i}\left( \sigma \right) \right| d\sigma +\left| e_{n,i}\right| \label{Lnew}
\end{equation}
where $\gamma _1$, $\gamma _2 \in \mathbb{R}$ are some positive bounding constants.
\end{lemma}
\begin{proof}
 The proof is very similar to that of the one given in \cite{Stepanyan-TNN}, however for the completeness of the presentation we have included it in Appendix \ref{App2}.
\end{proof}

\begin{lemma} \label{Lemma2}
Consider the term 
\begin{equation}
L \triangleq r^T\left( \bar{N}-\left( I_m+\Omega \right) C\text{Sgn}\left( e_n\right) \right)   \label{L}
\end{equation}
where $\Omega \left( t\right) $ defined in \eqref{bound} is strictly upper triangular and also is a function of desired trajectory and its time derivatives; thus it is bounded. Provided that the entries of the control gain $C$ are chosen to satisfy 
\begin{eqnarray}
C_m &\geq & \zeta _{\bar{N}_m}\left( 1+\frac{\gamma _2}{\alpha _m}\right) \label{Lgain1} \\
C_i &\geq & \left( \zeta _{\bar{N}_i}+\sum\limits_{j=i+1}^m\zeta _{\Omega _{i,j}}C_j\right) \left( 1+\frac{\gamma _2}{\alpha _i}\right) \text{ , } i=1,...,\left(m-1\right) \label{Lgain2}
\end{eqnarray}
then it can be concluded that
\begin{equation}
\int\limits_{t_0}^tL\left( \sigma \right) d\sigma \leq \zeta _L \label{Lproof}
\end{equation}
where $\zeta _L\in\mathbb{R}$ is a positive bounding constant defined as
\begin{equation}
\zeta _L \triangleq \gamma _1\sum\limits_{i=1}^{m-1}\sum\limits_{j=i+1}^m\zeta _{\Omega_{i,j}}C_j+\gamma _1\sum\limits_{i=1}^m\zeta _{\bar{N}_i} +\sum\limits_{i=1}^mC_i\left| e_{n,i}\left( t_0\right) \right| . \label{Lconstant}
\end{equation}
\end{lemma}
\begin{proof}
See Appendix \ref{App3}.
\end{proof}

\begin{theorem} \label{Theorem2}
\textbf{(Asymptotic convergence proof)} Given the  uncertain MIMO nonlinear system of the form \eqref{model1}, the continuous robust controller of \eqref{u} and \eqref{ub} ensures that all closed--loop signals remain bounded and the tracking error signals converges to zero asymptotically in the sense that
\begin{equation*}
e_{1}^{\left( i\right) } \to 0 \text{ as } t\rightarrow +\infty \text{ , } \forall i=0,\cdots ,n
\end{equation*}
 provided that $\alpha$ is chosen to satisfy \eqref{GainBound1}, the entries of $C$ are chosen to satisfy \eqref{Lgain1} and \eqref{Lgain2}, and $k_{d,i}$ and $k_p$ are chosen large enough.
\end{theorem}
\begin{proof}
See Appendix \ref{App4}.
\end{proof}

\section{Numerical Results}

In order to substantiate the theoretical results, numerical analysis has been carried out. Similar to \cite{ChenOFB} the performance and liability of the proposed nonlinear robust controller has been tested on a two--link robot manipulator system with coupling between the two links taken from \cite{slotine}. The equations of motion are given as \cite{ChenOFB}
\begin{equation}
\left[ \begin{array}{l} \tau _1^{*} \\ \tau _2^{*} \end{array} \right] = \left[ \begin{array}{ll} H_{11} & H_{12} \\ H_{12} & H_{22} \end{array} \right] 
\left[ \begin{array}{l} \ddot{q}_1 \\ \ddot{q}_2 \end{array} \right] 
+ \left[ \begin{array}{ll} -h\dot{q}_2 & -h(\dot{q}_1+\dot{q}_2) \\ -h\dot{q}_1 & 0 \end{array} \right] 
\left[ \begin{array}{l} \dot{q}_1 \\ \dot{q}_2 \end{array} \right] \label{m1}
\end{equation}
where $q_1\left(t\right)$, $q_2\left(t\right) \in \mathbb{R}$ denote the positions of the joint angles, and $H_{11}$, $H_{12}$, $H_{22}$ and $h$ are explicitly defined  as
\begin{eqnarray}
H_{11} &=& a_1+2a_3\cos q_2+2a_4\sin q_2 \\ 
H_{12} &=& a_2+a_3\cos q_2+a_4\sin q_2 \\ 
H_{22} &=& a_2 \\
h &=& a_3\sin q_2-a_4\sin q_2 . \label{m2}
\end{eqnarray}
where $a_1=4.42$, $a_2=0.97$, $a_3=1.04$ and $a_4=0.6$. The term $\left[ \begin{array}{cc} \tau _1^{*} & \tau _2^{*} \end{array} \right]^T$ in \eqref{m1} is obtained as
\begin{equation}
\left[ \begin{array}{l} \tau _1^{*} \\ \tau _2^{*} \end{array} \right] = \bar{\beta} (q_1,q_2)\left[ \begin{array}{ll} 1 & 1 \\ 0 & 1 \end{array} \right] 
\left[ \begin{array}{l} \tau _1 \\ \tau _2 \end{array} \right] \label{m3}
\end{equation}
where $\tau_1\left(t\right)$ and $\tau _2\left(t\right)$ are the control inputs, and $\bar{\beta} = H_{11}H_{22}-H_{12}^2\in \mathbb{R}$. The robot manipulator's initial positions have been set to $q\left(0\right)=[\begin{array}{ll} 10 & 10 \end{array}]^T(\deg )$ and $\dot{q}\left(0\right)=[ \begin{array}{ll} 0 & 0 \end{array} ]^T(\deg .s^{-1})$. The control objective is to make $q_1\left( t \right)$ and $q_2\left( t \right)$ follow a sinusoidal desired trajectory chosen as 
\begin{equation}
q_d(t)=(1-\exp (-0.3t^3))[ \begin{array}{ll} 30\sin (t) & 45\sin (t) \end{array} ]^T (\deg ) . \label{m4}
\end{equation}

After a rough tuning process the control gains are selected as follows 
\begin{equation}
\alpha =diag\{1,5\},~~K=diag\{175,125\} \text{ and \ }C=5.  \label{gains}
\end{equation}%
where $diag\left\{ \cdot \right\} $ is used to represent a diagonal matrix.

\begin{remark}
Though in the simulations the tuning process did not require the separation of gains $k_{d,i}$ and $k_p$ from $ K$ on an actual system this might be needed. In these cases, adjusting control gains  might become a bit tricky. We advice the users to adjust the entries of $C$ first, then based on $C$ adjust $k_p$, and finally work on $k_d$ based on $k_p$. 
\end{remark}

The link position tracking error is depicted in Figure \ref{error}, while the control input is shown in Figure \ref{torque}. Simulation results confirm that the proposed controller meets the position tracking objective and asymptotic tracking is achieved.  
\begin{figure}[ht]
\centering
\includegraphics[width=5.0in]{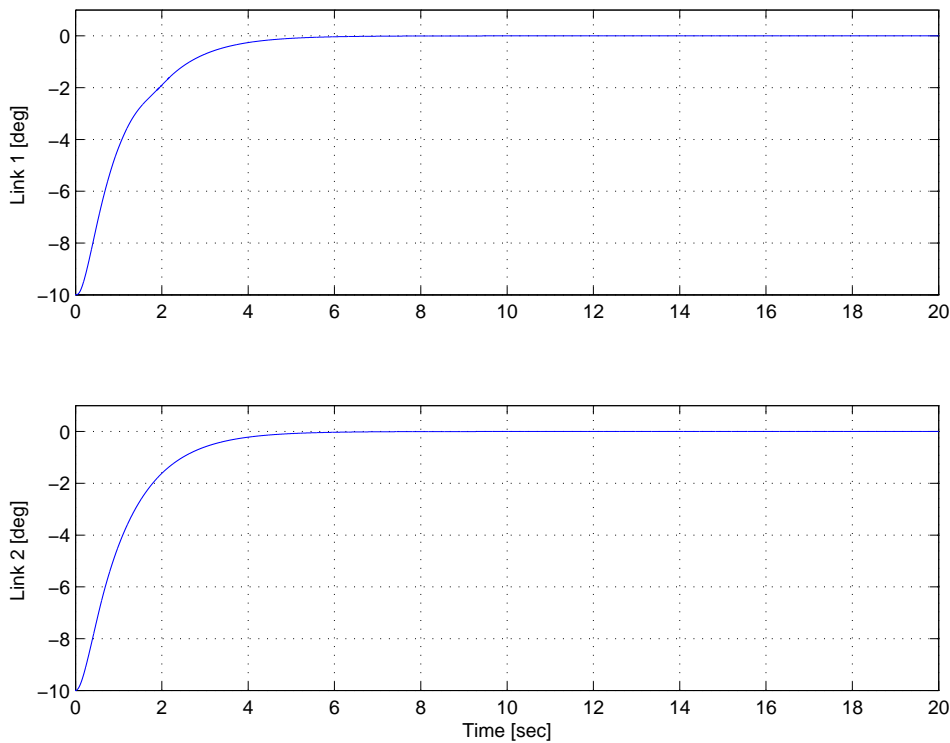}
\caption{Link Tracking Errors}
\label{error}
\end{figure}
\begin{figure}[ht]
\centering
\includegraphics[width=5.0in]{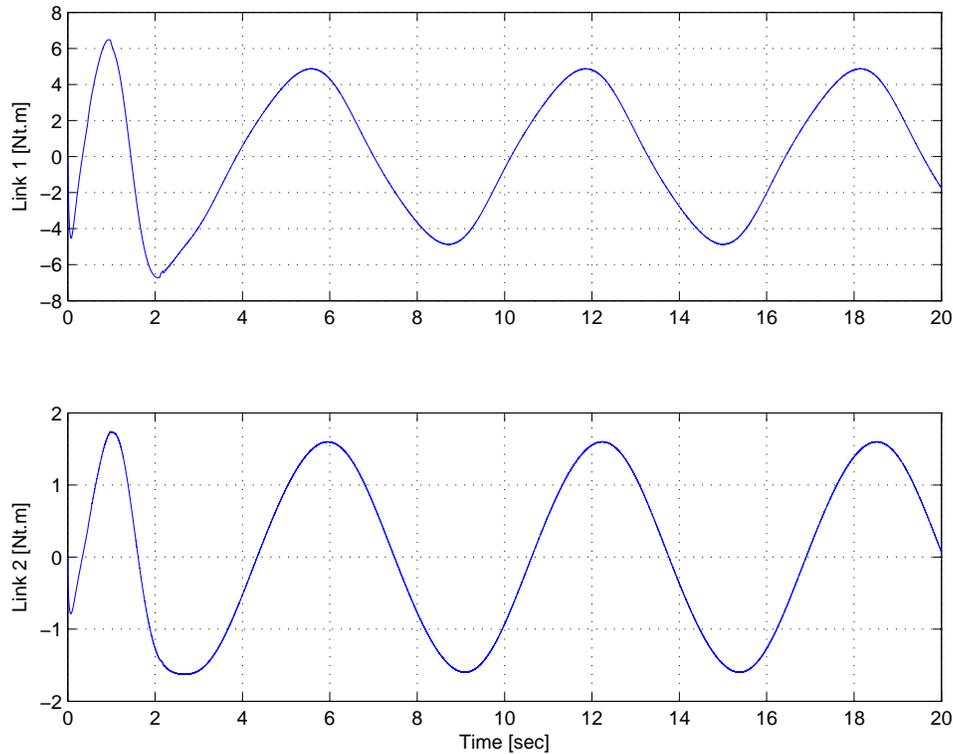}
\caption{Control Inputs}
\label{torque}
\end{figure}

\section{Conclusion}
In this work, for a class of uncertain MIMO systems having non--zero leading principle minors in their input gain matrix, a continuous nonlinear robust controller have been proposed. Via the use of Lyapunov based arguments in conjunction with an integral inequality we were able to obtain semi--global asymptotic tracking. The proposed controller can compensate the uncertainties though the system under a smoothness assumption on the system uncertainties. Simulation results are also included to illustrate the viability and the performance of the proposed method. 

\bibliographystyle{IEEEtran}
\bibliography{IEEEabrv,Refs}

\newpage
\appendices

\section{Proof of Theorem \protect\ref{Theorem1}}\label{App1}

\begin{proof} 
The non--negative function $V_{1}\left( z\right) \in \mathbb{R}$ is defined
as 
\begin{equation}
V_{1}\triangleq \frac{1}{2}\sum\limits_{i=1}^{n}e_{i}^{T}e_{i}+\frac{1}{2}%
r^{T}Mr.  \label{V}
\end{equation}%
By utilizing \eqref{M}, \eqref{V} can be bounded in the following manner 
\begin{equation}
\frac{1}{2}\min \left\{ 1,\underline{m}\right\} \left\Vert z\right\Vert
^{2}\leq V_{1}\left( z\right) \leq \frac{1}{2}\max \left\{ 1,\bar{m}\left(
\left\Vert z\right\Vert \right) \right\} \left\Vert z\right\Vert ^{2},
\label{Vbound}
\end{equation}%
where $z(t)$ was defined in \eqref{z}, and the terms  $\underline{m}$ , $\bar{m}\left(
\left\Vert z\right\Vert \right) $ were defined in \eqref{M}. Taking the  time derivative of \eqref{V} yields   
\begin{equation}
\dot{V}_{1}=\sum\limits_{i=1}^{n}e_{i}^{T}\dot{e}_{i}+r^{T}M\dot{r}+\frac{1}{%
2}r^{T}\dot{M}r.  \label{Vdot2}
\end{equation}%
The first term in the above expression can be written as follows 
\begin{eqnarray}
\sum\limits_{i=1}^{n}e_{i}^{T}\dot{e}_{i} &=&e_{1}^{T}\left(
e_{2}-e_{1}\right) +e_{2}^{T}\left( e_{3}-e_{2}-e_{1}\right)   \notag \\
&&+e_{3}^{T}\left( e_{4}-e_{3}-e_{2}\right) +...  \notag \\
&&+e_{n-1}^{T}\left( e_{n}-e_{n-1}-e_{n-2}\right) +e_{n}^{T}\left( r-\alpha
e_{n}\right)   \notag \\
&=&-\sum%
\limits_{i=1}^{n-1}e_{i}^{T}e_{i}+e_{n-1}^{T}e_{n}+e_{n}^{T}r-e_{n}^{T}%
\alpha e_{n}  \label{Vdot3}
\end{eqnarray}%
where \eqref{e2}--\eqref{en}, \eqref{r} were utilized. Substituting \eqref{close}--\eqref{u1b}, \eqref{u2}--\eqref{u2b} and \eqref{Vdot3} into \eqref{Vdot2} results in 
\begin{eqnarray}
\dot{V}_{1}
&=&-\sum%
\limits_{i=1}^{n-1}e_{i}^{T}e_{i}+e_{n-1}^{T}e_{n}+e_{n}^{T}r-e_{n}^{T}\alpha e_{n}  \notag \\
&&+r^{T}\left( -\frac{1}{2}\dot{M}r-e_{n}-Kr+\widetilde{N}+\bar{N}\right) \notag \\
&&-r^{T}\left[ \begin{array}{c} \Lambda +\Phi \\ 0 \end{array} \right] -r^{T}\left[ \begin{array}{c} \Psi \\ 0 \end{array} \right] 
-r^{T}\Theta +\frac{1}{2}r^{T}\dot{M}r  \label{Vdot4}
\end{eqnarray}%
which, after substituting the control gain matrix, can be rewritten as 
\begin{eqnarray}
\dot{V}_{1}
&=&-\sum\limits_{i=1}^{n-1}e_{i}^{T}e_{i}+e_{n-1}^{T}e_{n}-e_{n}^{T}\alpha e_{n}-r^{T}r  \notag \\
&&+\left[ r^{T}\widetilde{N}-k_{p}r^{T}r\right] +\left[ - \sum \limits_{i=1}^{m-1}r_{i}\left( \Lambda _{i}+\Psi _{i}+\Phi _{i}\right)
-\sum\limits_{i=1}^{m-1}k_{d,i}r_{i}^{2}\right] \notag \\
&&+r^{T}\bar{N}-r^{T}\Theta . \label{Vdot5}
\end{eqnarray}%
After completing the squares in bracketed terms, bounding $\bar{N}\left( t\right) $ and $\Theta \left( t\right) $ with constants, and utilizing $e_{n-1}^{T}e_{n}\leq 1/2\left\Vert e_{n-1}\right\Vert ^{2}+1/2\left\Vert e_{n}\right\Vert ^{2}$, we obtain 
\begin{eqnarray}
\dot{V}_{1} &\leq &-\sum_{i=1}^{n-2}\left\Vert e_{i}\right\Vert ^{2}-\frac{1}{2}\left\Vert e_{n-1}\right\Vert ^{2}-\left( \lambda _{\min }\left( \alpha
\right) -\frac{1}{2}\right) \left\Vert e_{n}\right\Vert ^{2}-r^{T}r  \notag \\
&&+\frac{\rho _{\widetilde{N}}^{2}\left( \left\Vert z\right\Vert \right) }{4k_{p}}\left\Vert z\right\Vert ^{2}+\sum\limits_{i=1}^{m-1}\frac{\rho_{i}^{2}}{4k_{d,i}} \left\Vert z\right\Vert ^{2}+\left\Vert r\right\Vert \zeta _{\bar{N}}+\left\Vert r\right\Vert \zeta _{\Theta }  \label{Vdot1A}
\end{eqnarray}%
which can then be rearranged as 
\begin{equation}
\dot{V}_{1}\leq -\left( \lambda _{1}-\frac{\rho _{\widetilde{N}}^{2}\left(
\left\Vert z\right\Vert \right) }{4k_{p}}-\sum\limits_{i=1}^{m-1}\frac{\rho
_{i}^{2}\left( \left\Vert z\right\Vert \right) }{4k_{d,i}}\right)
\left\Vert z\right\Vert ^{2}+\delta \varepsilon ^{2}  \label{Vdot2A}
\end{equation}%
where $\lambda _{1}\triangleq \min \left\{ \frac{1}{2},\lambda _{\min
}\left( \alpha \right) -\frac{1}{2},1-\frac{1}{4\delta }\right\} $, $\delta
\in \mathbb{R}$ is a positive bounding constant, $\varepsilon \triangleq
\zeta _{\bar{N}}+\zeta _{\Theta }$, and $\left\Vert r\right\Vert \varepsilon
\leq \frac{1}{4\delta }\left\Vert r\right\Vert ^{2}+\delta \varepsilon ^{2}$
was utilized. When the  controller gains $k_{d,i}$ and $k_{p}$ are selected large enough
(compared to the initial conditions of $z\left( t\right) $), and utilizing  \eqref{Vbound},  the following inequality can be obtained 
\begin{equation}
\dot{V}_{1}\leq -\beta _{1} V_1+\delta \varepsilon
^{2}  \label{V1Bound1}
\end{equation}%
where $\beta _{1}\in \mathbb{R}$ is a positive constant. From \eqref{V}, and \eqref{V1Bound1}, we can conclude that $V_1(t)$, therefore $e_{i}\left(
t\right) $  for  $ i=1,...,n$ and  $r\left( t\right) $ are uniformly ultimately bounded.
\end{proof}

\section{Proof of Lemma \ref{Lemma1}} \label{App2}

\begin{proof}
First, we note that if $e_{n,i}(t)\equiv 0$ on some interval, then $\dot{e}_{n,i}(t)\equiv 0$ on the same interval, and the inequality \eqref{Lnew} yields this qualification. Therefore, without loss of generality, we assume that $e_{n,i}\left(t\right)$ is absolutely greater than zero on
the interval of $\left[ t_0,t\right] $. Let $T\in \left[ t_0,t\right) $ be the last instant of time when $\dot{e}_{n,i}\left(t\right)$ changes sign. Then, on the interval $\left[ T,t\right] $, $\dot{e}_{n,i}\left(t\right)$ has a constant sign, hence
\begin{equation}
\int\nolimits_T^t\left| \dot{e}_{n,i}\left(\sigma\right) \right| d\sigma = \left| \int\nolimits_T^t\dot{e}_{n,i}\left(\sigma\right) d\sigma \right|
= \left| e_{n,i}\left(t\right)-e_{n,i}\left(T\right)\right| .  \label{p1}
\end{equation}
From the boundedness of the function $\dot{e}_{n,i}(t)$, it follows that there exist a constant $\gamma >0$ such that $\left| \dot{e}_{n,i}\left(t\right)\right| \leq \gamma $, therefore
\begin{equation}
\int\nolimits_{t_0}^T\left| \dot{e}_{n,i}\left(\sigma\right)\right| d\sigma \leq \gamma \left(T-t_0\right).  \label{p2}
\end{equation}
On the other hand, we obtain the following equality from the application of the Mean Value Theorem \cite{khalil}
\begin{equation}
\int\nolimits_{t_0}^T\left| e_{n,i}\left(\sigma\right)\right| d\sigma =e_{n,i_{*}}\left(T-t_0\right) \label{p3}
\end{equation}
where $e_{n,i_{*}}$ is some intermediate value of $\left| e_{n,i}(t)\right| $ on the interval $\left[ t_0,T\right] $. By assumption, $e_{n,i_{*}}$ is
bounded away from zero. Therefore we can conclude as follows by using inequality \eqref{p2} and equality \eqref{p3}
\begin{equation}
\int\nolimits_{t_0}^T\left| \dot{e}_{n,i}\left(\sigma\right)\right| d\sigma \leq \gamma _2\int\nolimits_{t_0}^T\left| e_{n,i}\left(\sigma\right)\right| d\sigma \label{p4}
\end{equation}
where $\gamma _2= \gamma /e_{n,i_{*}} $. Combining the relationships in \eqref{p1} and \eqref{p4}, we can write
\begin{equation}
\int\nolimits_{t_0}^t\left| \dot{e}_{n,i}\left(\sigma\right) \right| d\sigma \leq \left| e_{n,i}\left(t\right)\right| +\left| e_{n,i}\left(T\right)\right| + \gamma_2\int\nolimits_{t_0}^T\left| e_{n,i}\left(\sigma\right)\right| d\sigma  \label{p5}
\end{equation}
which yields the inequality \eqref{Lnew} with definition $\gamma _1 \triangleq \sup\left| e_{n,i} \left(T\right)\right| $. 
\end{proof}

\section{Proof of Lemma \ref{Lemma2}} \label{App3}

\begin{proof}
We start our analysis by integrating \eqref{L} in time from $t_0$ to $t$
\begin{eqnarray}
\int\limits_{t_0}^tL\left( \sigma \right) d\sigma 
&=&\int\limits_{t_0}^te_n^T\left(\sigma\right)\alpha^T\left(\bar{N}\left(\sigma\right)-C\text{Sgn}\left(e_n\left(\sigma\right)\right)\right)d\sigma \nonumber\\
&&-\int\limits_{t_0}^te_n^T\left(\sigma\right)\alpha^T\Omega\left(\sigma\right)C\text{Sgn}\left( e_n\left( \sigma \right) \right) d\sigma \nonumber \\
&&+\int\limits_{t_0}^t\dot{e}_n^T\left( \sigma \right) \bar{N}\left( \sigma\right) d\sigma 
-\int\limits_{t_0}^t\dot{e}_n^T\left( \sigma \right) \Omega \left( \sigma \right) C\text{Sgn}\left( e_n\left( \sigma \right) \right) d\sigma \nonumber \\
&&-\int\limits_{t_0}^t\dot{e}_n^T\left( \sigma \right) C\text{Sgn}\left( e_n\left( \sigma \right) \right) d\sigma   \label{L1}
\end{eqnarray}
where \eqref{r} was utilized. To ease the presentation, we will consider each term on the right--hand side of \eqref{L1} separately:
\subsection*{The First Term:}
\begin{eqnarray}
\int\limits_{t_0}^te_n^T\left( \sigma \right) \alpha ^T\left( \bar{N}\left( \sigma \right) 
-C\text{Sgn}\left( e_n\left( \sigma \right) \right) \right) d\sigma  
&=& \int\limits_{t_0}^t\sum\limits_{i=1}^m \alpha _i e_{n,i}\left( \sigma \right) \left( \bar{N}_i\left( \sigma \right) 
-C_i\text{sgn}\left(e_{n,i}\left( \sigma \right) \right) \right) d\sigma   \nonumber \\
&\leq &\sum\limits_{i=1}^m\alpha _i\left( \zeta _{\bar{N}_i}-C_i\right) \int\limits_{t_0}^t\left| e_{n,i}\left( \sigma \right) \right| d\sigma \label{Lterm1}
\end{eqnarray}
\subsection*{The Second Term:}
\begin{eqnarray}
-\int\limits_{t_0}^te_n^T\left( \sigma \right) \alpha ^T\Omega \left( \sigma \right) C\text{Sgn}\left( e_n\left( \sigma \right) \right) d\sigma 
&=& -\int\limits_{t_0}^t\sum\limits_{i=1}^{m-1}\alpha _ie_{n,i}\left( \sigma \right) \sum\limits_{j=i+1}^m C_j \Omega _{i,j} \left( \sigma \right) 
\text{sgn}\left( e_{n,j}\left( \sigma \right) \right) d\sigma \nonumber \\
&\leq &\sum\limits_{i=1}^{m-1}\sum\limits_{j=i+1}^m\alpha _i C_j \zeta_{\Omega_{i,j}}\int\limits_{t_0}^t\left| e_{n,i}\left(\sigma\right)\right|d\sigma
\label{Lterm2}
\end{eqnarray}
\subsection*{The Third Term:}
\begin{eqnarray}
\int\limits_{t_0}^t\dot{e}_n^T\left( \sigma \right) \bar{N}\left( \sigma \right) d\sigma  &=&\int\limits_{t_0}^t\sum\limits_{i=1}^m\dot{e}_{n,i}^T\left( \sigma \right) \bar{N}_i\left( \sigma \right) d\sigma \nonumber \\
&\leq &\sum\limits_{i=1}^m\zeta _{\bar{N}_i}\int\limits_{t_0}^t\left| \dot{e}_{n,i}\left( \sigma \right) \right| d\sigma \nonumber \\
&\leq &\sum\limits_{i=1}^m\zeta _{\bar{N}_i}\left( \gamma _1+\gamma_2\int\limits_{t_0}^t\left| e_{n,i}\left( \sigma \right) \right| d\sigma 
+\left| e_{n,i}\right| \right) \label{Lterm3}
\end{eqnarray}
\subsection*{The Fourth Term:}
\begin{eqnarray}
-\int\limits_{t_0}^t\dot{e}_n^T\left( \sigma \right) \Omega \left( \sigma \right) C\text{Sgn}\left( e_n\left( \sigma \right) \right) d\sigma 
&=& -\int\limits_{t_0}^t\sum\limits_{i=1}^{m-1}\dot{e}_{n,i}\left( \sigma \right) \sum\limits_{j=i+1}^m C_j \Omega _{i,j}\left( \sigma \right) \text{sgn}
\left( e_{n,j}\left( \sigma \right) \right) d\sigma \nonumber \\
&\leq &\sum\limits_{i=1}^{m-1}\sum\limits_{j=i+1}^m C_j \zeta _{\Omega_{i,j}}\int\limits_{t_0}^t\left| \dot{e}_{n,i}\left( \sigma \right) \right| d\sigma  \label{Lterm4} \\
&\leq &\sum\limits_{i=1}^{m-1}\sum\limits_{j=i+1}^m C_j \zeta _{\Omega_{i,j}}\left( \gamma _1+\gamma _2\int\limits_{t_0}^t\left| e_{n,i}\left( \sigma \right) \right| d\sigma +\left| e_{n,i}\right| \right)  \nonumber
\end{eqnarray}
\subsection*{The Fifth Term:}
\begin{eqnarray}
-\int\limits_{t_0}^t\dot{e}_n^T\left( \sigma \right) C\text{Sgn}\left( e_n\left( \sigma \right) \right) d\sigma 
&=& -\int\limits_{t_0}^t\sum\limits_{i=1}^m C_i \dot{e}_{n,i}\left( \sigma \right) \text{sgn}\left( e_{n,i}\left( \sigma \right) \right) d\sigma
\nonumber \\
&=& -\sum\limits_{i=1}^mC_i\int\limits_{t_0}^t\text{sgn}\left( e_{n,i}\left(\sigma \right) \right) d\left( e_{n,i}\right) \nonumber \\
&=&-\sum\limits_{i=1}^mC_i\int\limits_{t_0}^td\left( \left| e_{n,i}\right| \right) \nonumber \\
&=&-\sum\limits_{i=1}^mC_i\left| e_{n,i}\left( t\right) \right| + \sum\limits_{i=1}^mC_i\left| e_{n,i}\left( t_0\right) \right| . \label{Lterm5}
\end{eqnarray}
It is noted that, the result of Lemma \ref{Lemma1} was utilized to obtain \eqref{Lterm3} and \eqref{Lterm4}. After combining the upper bounds in \eqref{Lterm1}--\eqref{Lterm5}, we obtain 
\begin{eqnarray}
\int\limits_{t_0}^tL\left( \sigma \right) d\sigma  &\leq & 
\sum\limits_{i=1}^{m-1}\alpha _i\left[ \left(1 +\frac{\gamma _2}{\alpha _i}\right) \left( \zeta _{\bar{N}_i} + \sum\limits_{j=i+1}^m \zeta_{\Omega_{i,j}}C_j \right) - C_i\right] \int\limits_{t_0}^t\left| e_{n,i}\left( \sigma \right) \right| d\sigma  \nonumber \\
&&+\alpha _m\left[ \left(1 +\frac{\gamma _2}{\alpha _m}\right)\zeta _{\bar{N}_m}-C_m\right] \int\limits_{t_0}^t\left| e_{n,m}\left( \sigma \right) \right| d\sigma \nonumber \\
&&+\left( \zeta _{\bar{N}_m}-C_m\right) \left| e_{n,m}\right|   \nonumber \\
&&+\sum\limits_{i=1}^{m-1}\left( \zeta _{\bar{N}_i}+\sum\limits_{j=i+1}^m\zeta_{\Omega _{i,j}}C_j-C_i\right) \left| e_{n,i}\right|   \nonumber \\
&&+\gamma _1\sum\limits_{i=1}^{m-1}\sum\limits_{j=i+1}^m\zeta _{\Omega_{i,j}}C_j+\gamma _1\sum\limits_{i=1}^m\zeta _{\bar{N}_i}
+\sum\limits_{i=1}^mC_i\left| e_{n,i}\left( t_0\right) \right| . \label{Lint3}
\end{eqnarray}
Based on \eqref{Lint3}, we first choose $C_m$ to satisfy \eqref{Lgain1} to make second and third expressions on the right--hand side negative,
we next choose $C_i$ starting from $\left( m-1\right) $ with a decreasing order to satisfy \eqref{Lgain2} to make first and fourth expressions on the right--hand side negative,
and finally, we utilized the definition of $\zeta_L$ in \eqref{Lconstant} to obtain \eqref{Lproof}, thus completing the proof of Lemma \ref{Lemma2}.
\end{proof}

\section{Proof of Theorem \protect\ref{Theorem2}}

\label{App4}

\begin{proof}
Let the auxiliary function $P\left( {t}\right) \in \mathbb{R}$ be defined as follows 
\begin{equation}
P\triangleq \zeta _{L}-\int_{t_{0}}^{t}L\left( \sigma \right) d\sigma.
\label{p11}
\end{equation}%
where the terms $\zeta _{L}$ and $ L(t)$ were defined in \eqref{Lconstant} and \eqref{L}, respectively, when the entries of the control gain matrix $C$ are chosen to satisfy \eqref{Lgain1} and \eqref{Lgain2}, from the proof of Lemma \ref{Lemma2} given in Appendix \ref{App3}, we can conclude that $P\left( t\right) $ is non--negative. At this stage, consider the Lyapunov function, denoted by $V\left( s,t\right) \in \mathbb{R}$, defined as follows
\begin{equation}
V\triangleq V_{1}+P  \label{Lyap}
\end{equation}%
where $s\left( t\right) \in \mathbb{R}^{\left[ \left( n+1\right) m + 1\right] \times 1}$ is defined as 
\begin{equation}
s\triangleq \left[ 
\begin{array}{ll}
z^{T} & \sqrt{P}%
\end{array}%
\right] ^{T}.  \label{ss}
\end{equation}%
By utilizing \eqref{M}, \eqref{Lyap} can be upper and lower bounded in the following form 
\begin{equation}
W_{1}\left( s\right) \leq V\left( s,t\right) \leq W_{2}\left( s\right) 
\label{V2bound}
\end{equation}%
where $W_{1}\left( s\right) $, $W_{2}\left( s\right) \in \mathbb{R}$ are
defined as 
\begin{equation}
W_{1}\triangleq \lambda _{2}\left\Vert s\right\Vert ^{2}\text{,  }%
W_{2}\triangleq \lambda _{3}\left( \left\Vert s\right\Vert \right)
\left\Vert s\right\Vert ^{2}  \label{V2bound2}
\end{equation}%
with $\lambda_{2} \triangleq \frac{1}{2}\min \left\{ 1,\underline{m}\right\} $ and 
$ \lambda_{3} \triangleq \max \left\{ 1,\frac{1}{2}\bar{m}\left( \left\Vert z\right\Vert \right) \right\} $ .

Taking the time derivative of $V\left( t\right) $, utilizing the time derivative of \eqref{Lproof}, canceling common terms and following similar steps to that of proof of Theorem \ref{Theorem1}  yields   
\begin{eqnarray}
\dot{V}
&=&-\sum\limits_{i=1}^{n-1}e_{i}^{T}e_{i}+e_{n-1}^{T}e_{n}-e_{n}^{T}\alpha
e_{n}-r^{T}r  \notag \\
&&+\left[ r^{T}\widetilde{N}-k_{p}r^{T}r\right] +\left[ -\sum\limits_{i=1}^{m-1}r_{i}\left( \Lambda _{i}+\Psi _{i}+\Phi _{i}\right)
-\sum\limits_{i=1}^{m-1}k_{d,i}r_{i}^{2}\right]   \label{V2}
\end{eqnarray}
which can be rearranged to have the following form 
\begin{eqnarray}
\dot{V} &\leq &-\sum_{i=1}^{n-2}\left\Vert e_{i}\right\Vert ^{2}-\frac{1%
}{2}\left\Vert e_{n-1}\right\Vert ^{2}-\left( \lambda _{\min }\left( \alpha
\right) -\frac{1}{2}\right) \left\Vert e_{n}\right\Vert ^{2}-r^{T}r  \notag
\\
&&+\frac{\rho _{\widetilde{N}}^{2}\left( \left\Vert z\right\Vert \right) }{%
4k_{p}}\left\Vert z\right\Vert ^{2}+\sum\limits_{i=1}^{m-1}\frac{\rho
_{i}^{2}\left( \left\Vert z\right\Vert \right) }{4k_{d,i}}\left\Vert
z\right\Vert ^{2}  \label{V2c} \\
&\leq &-\left( \lambda _{4}-\frac{\rho _{\widetilde{N}}^{2}\left( \left\Vert
z\right\Vert \right) }{4k_{p}}-\sum\limits_{i=1}^{m-1}\frac{\rho
_{i}^{2}\left( \left\Vert z\right\Vert \right) }{4k_{d,i}}\right) \left\Vert
z\right\Vert ^{2}  \label{V2d}
\end{eqnarray}%
where $\lambda _{4}\triangleq \min \left\{ \frac{1}{2},\lambda _{\min
}\left( \alpha \right) -\frac{1}{2}\right\} $. When the controller gains $k_{p}$ and $k_{d,i}$ for $i=1,2,...,\left(
n-1\right) $ are selected large enough such that the regions defined by $\mathcal{D}_{z}\triangleq \left\{ z:\left\Vert z\right\Vert \leq \mathcal{R}
\right\} $ and $\mathcal{D}_{s}\triangleq \left\{ s:\left\Vert s\right\Vert\leq \mathcal{R}\right\} $ with $\mathcal{R}$ defined as  
\begin{equation}
\mathcal{R=}\min \left\{ \rho _{\widetilde{N}}^{-1}\left( 2\sqrt{k_{p}\frac{1-\beta}{m}}\right) ,\rho _{i}^{-1}\left( 2\sqrt{k_{d,i}\frac{1-\beta }{m}}\right) \right\} \text{ for \ \ }i=1,2,...,\left( m-1\right) 
\label{Regiondef}
\end{equation}%
are non--empty, from \eqref{V2d} and the definition of $s$, one can restate 
\begin{equation}
\dot{V}\leq -\beta \left\Vert z\right\Vert ^{2}=-W\left( s\right) ,\forall s\in \mathcal{D}_{s}  \label{someVdot}
\end{equation}%
where $\beta \in \mathbb{R}$ is a positive constant that satisfies $0\leq \beta < 1$. From the definition of \eqref{Lyap} and \eqref{someVdot}, it is obvious that and $V \left( t \right) \in \mathcal{L}_{\infty }$, also from the proof of Theorem \ref{Theorem1} and outcome of standart linear analysis
methods, we can conclude that all signal in the closed--loop error system are bounded and furthermore, from the boundedness of $\dot{W}(s)$, we can state $W\left( s\right) $ is uniformly continuous.

Based on the definition of $\mathcal{D}_{s}$, another region, $%
\mathcal{S}$, can be defined in the following form%
\begin{equation}
\begin{array}{cl}
\mathcal{S}\triangleq  & \left\{ s\in \mathcal{D}_{s}:W_{2}\left( s\right) <\lambda _{3}\left( \rho _{\widetilde{N}}^{-1}\left(2\sqrt{k_{p} \frac{1-\beta }{m}}\right) \right) ^{2}\right\}  \\ 
& \cap \left\{ s\in \mathcal{D}_{s}:W_{2}\left( s\right) <\lambda _{3}\left( \rho _{1}^{-1}\left( 2\sqrt{k_{d,1}\frac{1-\beta}{m}}\right) \right)^{2}\right\}  \\ 
& \cap \vdots \\ 
& \cap \left\{ s\in \mathcal{D}_{s}:W_{2}\left( s\right) <\lambda _{3}\left( \rho _{\left( m-1\right) }^{-1}\left( 2\sqrt{k_{d,\left( m-1\right) }\frac{%
1-\beta}{m}}\right) \right) ^{2}\right\} .
\end{array}
\label{S}
\end{equation}
A direct application of Theorem 8.4 in \cite{khalil} can be used to prove that $\left\Vert z\left( t\right) \right\Vert \rightarrow 0$ as $t\rightarrow +\infty $ $\forall s\left( t_{0}\right) \in \mathcal{S}$. Based on the definition of $z\left( t\right) $, it is easy to show that $\left\Vert e_{i}\left( t\right) \right\Vert ,\left\Vert r\left( t\right) \right\Vert \rightarrow 0$ as $t\rightarrow +\infty $ $\forall s\left( t_{0}\right) \in \mathcal{S}$, $i=1,2,...,n$. From \eqref{r}, it is clear that $\left\Vert \dot{e}_{n}\left( t\right) \right\Vert \rightarrow 0$ as $t\rightarrow +\infty $ $\forall s\left( t_{0}\right) \in \mathcal{S}$. By utilizing \eqref{ei} recursively, it can be proven that $\left\Vert e_{1}^{\left( i\right) }\left( t\right) \right\Vert \rightarrow 0$ as $t\rightarrow +\infty $, $i=1,2,...,n$ $\forall s\left( t_{0}\right) \in \mathcal{S}$. Note that the region of attraction can be made arbitrarily large to include any initial conditions by choosing the controller gains $k_{p}$ and $k_{d,i}$ for $i=1,2,...,\left( m-1\right) $.  This fact implies that the stability result obtained by proposed method is semi--global.
\end{proof} 

\end{document}